\documentclass[11pt,twoside]{article}
\usepackage{amsfonts}
\usepackage{amsmath}
\usepackage{amssymb}
\usepackage{multicol}
\usepackage{graphics}
\usepackage{stmaryrd}
\usepackage{cite}
\usepackage{epsfig}

\newtheorem{theorem}{Theorem}[section]
\newtheorem{lemma}[theorem]{Lemma}

\newtheorem{definition}[theorem]{Definition}
\newtheorem{remark}[theorem]{Remark}
\newtheorem{proposition}[theorem]{Proposition}
\numberwithin{equation}{section}
\newenvironment{proof}[1][Proof]{\noindent\textbf{#1.} }{\hfill $\Box$}
\allowdisplaybreaks

 \makeatletter
\begin{document}

\author{Guo Lin\thanks{E-mail: ling@lzu.edu.cn.} \\
{School of Mathematics and Statistics, Lanzhou University,}\\
{Lanzhou, Gansu 730000, People's Republic of China}}
\title{\textbf{Asymptotic Spreading Fastened by Inter-Specific Coupled Nonlinearities:
a Cooperative System}}\maketitle

\begin{abstract}
This paper is concerned with the asymptotic spreading of a
Lotka-Volterra cooperative system. By using the theory of asymptotic spreading of nonautonomous equations, the asymptotic speeds of spreading of unknown functions formulated by a coupled system are estimated.
Our results imply that the asymptotic spreading of one species can
be significantly fastened by introducing a mutual species, which
indicates the role of cooperation described by the coupled
nonlinearities.


\end{abstract}

\begin{center}
Accepted by Phisica D with DOI: 10.1016/j.physd.2011.12.007
\end{center}

\newpage

\section{Introduction}
\noindent

In this paper, we consider the propagation of the following
 diffusion system
\begin{equation}
\begin{cases}
\frac{\partial u_1(t,x)}{\partial t}=d_1\Delta
u_1(t,x)+r_1u_1(t,x)\left[
1-u_1(t,x)+b_1u_2(t,x)\right] , \\
\frac{\partial u_2(t,x)}{\partial t}=d_2\Delta
u_2(t,x)+r_2u_2(t,x)\left[ 1-u_2(t,x)+b_2u_1(t,x)\right],
\end{cases}\label{0}
\end{equation}
in which $u_1(t,x),u_2(t,x)$ denote the densities of two
collaborators at time $t>0$ and location $x\in \mathbb{R}$ in population dynamics,   all the parameters are positive
and $b_1b_2<1$ such that \eqref{0} has four spatial homogeneous
steady states (for short, four equilibria)
$$(0,0),\left(1,0\right), \left(0,1\right)$$
and $K=(k_1,k_2)$ defined by
\[
(k_1,k_2)=\left(\frac{1+b_1}{1-b_1b_2},\frac{1+b_2}{1-b_1b_2}
\right).
\]
It is well known that $(k_1,k_2)$ is asymptotic stable while $(0,0),\left(1,0\right), \left(0,1\right)$ are unstable in the corresponding spatial homogeneous system of \eqref{0}.

Recently, Li et al. \cite{liwe} have investigated the
traveling wavefronts of \eqref{0} by using the theory established by Weinberger
et al. \cite{weinberger}, and the authors proved that the minimal wave
speed of traveling wavefronts of \eqref{0} can be linearly
determinate (see \cite{bo,moli}). In the modeling of population invasions (see Shigesada and  Kawasaki \cite{shi} for many important historic records), the linear determinacy
indicates that the minimal wave speed can be formulated by the
parameters appearing in the system linearized at the invadable
equilibrium which often is unstable in the corresponding kinetic
system. In population dynamics, besides the minimal wave speeds of
traveling wavefronts, the asymptotic speeds of spreading may also be linearly determinate, especially for
the \emph{scalar} equations, we refer to Aronson and Weinberger
\cite{aron1}, van den Bosch \cite{bo}, Diekmann
\cite{diekmann,diek}, Hsu and Zhao \cite{hsuzhao}, Lui
\cite{lui1,lui2}, Mollison \cite{moli}, Thieme \cite{th1}, Thieme
and Zhao \cite{thieme} for some examples.

However, the nonlinearities in equations/systems often give
expression to the inter- or intra-specific actions in population
dynamics. Intuitively, the effect of nonlinearities should be
reflected by many dynamical properties including the asymptotic speeds of spreading. Namely, the linear determinacy of minimal wave speed of traveling wavefronts and
asymptotic speeds of spreading cannot be true for all nonlinear models.  For autonomous scalar equations, a famous counter example of linear determinacy is
\begin{equation}\label{e}
\frac{\partial u(t,x)}{\partial t}=\Delta u (t,x)
+u(t,x)(1-u(t,x))(1+\nu u(t,x)),
\end{equation}
where $\nu>-1$ is a constant that does not appear in the following linearized system
\[
\frac{\partial u(t,x)}{\partial t}=\Delta u (t,x)
+u(t,x),
 \]
and we refer to Hadeler and Rothe \cite{hr} for
precise results on its asymptotic speed of spreading,  which is not linearly
determinate for  $\nu>2$. Moreover, some results on asymptotic spreading have also been obtained for coupled diffusion systems with multi equilibria, which formulates the role of
inter-specific coupled nonlinearities, see Lin
et al. \cite{llr1} and  Weinberger et al. \cite{wll} for two examples of
integral-difference equations, Lin \cite{lin} for a predator-prey reaction-diffusion system.

For reader's convenience, we first give the following definition.
\begin{definition}
\label{d1.1}{\rm Assume that $u(t,x)$ is a nonnegative function for
$x\in \mathbb{R},$ $t>0$. Then $c_{*}$ is called the
\textbf{asymptotic speed of spreading} of $u(t,x)$ if
\begin{description}
\item[a)] $\lim_{t\rightarrow \infty}
\sup_{|x|>(c_*+\epsilon)t}u(t,x)=0$ for any given $\epsilon >0$;
\item[b)] $\liminf_{t\rightarrow \infty}
\inf_{|x|<(c_*-\epsilon)t}u(t,x)>0$ for any given $\epsilon \in
(0,c_*)$.
\end{description}}
\end{definition}
Clearly, the asymptotic speed of spreading states the
observed phenomena if an observer were to move to the right or left
at a fixed speed \cite{weinberger}. Biologically, it also describes
the speed at which the geographic range of the new population
expands \cite{hsuzhao}. Therefore, it becomes a very important index formulating
the spatial propagation of ecological communities. At the same time, it is possible that the asymptotic speed of spreading of a nonnegative
function is not a positive constant in the above limit sense, see Berestycki et al. \cite{be2} for some examples. When the asymptotic speed of spreading is not a constant, its lower bounds and upper bounds in \cite[Section 1.8]{be} are still useful because these can describe and estimate the success of biological invasions.

If an irreducible cooperative system has just two equilibria in
the interesting interval, it is very likely that all the unknown
functions have the same asymptotic speed of spreading coincided with the linear determinacy, see some results by Liang
and Zhao \cite{liang}, Lui \cite{lui1,lui2}. In particular, when \eqref{0} is concerned, Li et al. \cite[Example 4.1]{liwe} studied the propagation modes when one species is the aboriginal and the other is the invader, namely, the
interesting interval is
\[[1,k_1]\times [0,k_2 ] \text{ or }[0,k_1]\times [1,k_2],\] on which the system has \emph{no} other equilibria, and can also be studied by \cite{liang,lui1,lui2}.

In this paper, we consider the asymptotic spreading of \eqref{0}
when both species are invaders, namely, $(0,0)$ will be the invadable equilibrium and the interesting interval
will be $[0,k_1]\times [0,k_2],$ on which \eqref{0} has four equilibria
such that we cannot use the theory of \cite{liang,lui1,lui2}.
To obtain some estimates on asymptotic spreading, the abstract results
developed by Berestycki et al. \cite{be} will be applied, and the lower bounds of
asymptotic speeds of spreading will be estimated. More precisely, we first give some properties of $u_1,$ then we regard the second equation of  \eqref{0} as a nonautonomous equation and establish some conclusions by \cite{be}. Our
results imply that: (1) The nonlinearities described the inter-specific actions may play an important role
in asymptotic spreading such that the asymptotic spreading of one species is faster than the case that the inter-specific actions disappear; (2) It is necessary to use different
indices to formulate  the asymptotic spreading of each unknown functions
if the system has multi equilibria. Moreover, our results answer the nonexistence of traveling wave solutions of \eqref{0}, which also develops the theory of traveling wave solutions in Lin et al. \cite{llm}.

In Section 2, we shall give some preliminaries, including a classical conclusion of Fisher equation and an
important result established by Berestycki et al. \cite{be}. Then we
shall show some estimates on the asymptotic spreading of \eqref{0}
if both species are invaders, which are also applied to the study of the corresponding traveling wave solutions. In the last section, further
discussion is provided to illustrate our conclusions.

\section{Preliminaries}
\noindent

We first present some results of the following Fisher's equation
\begin{equation}\label{1}
\begin{cases}
\frac{\partial z(t,x)}{\partial t}=d \Delta z(t,x)+rz (t,x)\left[
1-z(t,x)/K\right] ,\\
z(0,x)=z(x),
\end{cases}
\end{equation}
in which all the parameters are positive and  $z(x)>0$ is a uniform
 continuous and bounded function. Due to the theory of asymptotic spreading
established by Aronson and Weinberger \cite{aron1}, we have the following result.
\begin{lemma}\label{l1.1}
Assume that $z(t,x)$ is defined by \eqref{1} and $\epsilon \in (0, 2\sqrt{dr})$ holds. Then  \[\lim_{t\to
\infty}\inf_{ |x|< (2\sqrt{dr}-\epsilon)t}z(t,x)=K.\]
 Moreover, if $z(x)$ admits compact support, then
 \[\lim_{t\to
\infty}\sup_{ |x|> (2\sqrt{dr}+\epsilon)t}z(t,x)=0.\]
\end{lemma}

For \eqref{1}, the following comparison principle is also true (see
Ye and Li \cite{yeli}).
\begin{lemma}\label{com}
Assume that $\overline{z}(t,x)\ge 0, x\in \mathbb{R}, t>0,$
satisfies
\begin{equation*}
\begin{cases}
\frac{\partial \overline{z}(t,x)}{\partial t}\ge (\le) d \Delta
\overline{z}(t,x)+r\overline{z} (t,x)\left[
1-\overline{z}(t,x)/K\right] ,\\
\overline{z}(0,x)\ge (\le)z(x).
\end{cases}
\end{equation*}
Then $\overline{z}(t,x) \ge (\le){z}(t,x),$ where $z(t,x)$ is
defined by \eqref{1}.
\end{lemma}

For the system \eqref{0}, we also give the following comparison
principle (one also refers to Pao \cite{pao}, Smoller \cite{smoller},
Ye and Li \cite{yeli} for more details).
\begin{lemma}\label{lecom}
Let $(u_1(t,x),u_2(t,x))$ be defined by
\begin{equation*}
\begin{cases}
\frac{\partial u_1(t,x)}{\partial t}=d_1\Delta
u_1(t,x)+r_1u_1(t,x)\left[
1-u_1(t,x)+b_1u_2(t,x)\right] , \\
\frac{\partial u_2(t,x)}{\partial t}=d_2\Delta
u_2(t,x)+r_2u_2(t,x)\left[ 1-u_2(t,x)+b_2u_1(t,x)\right],\\
u_1(0,x)=u_1(x), u_2(0,x)=u_2(x),
\end{cases}
\end{equation*}
where $u_1(x)>0, u_2(x)>0$ are uniformly continuous and bounded.
If $(z_1(t,x),z_2(t,x))\ge (0,0)$ is uniformly continuous and
bounded for $(t,x)\in (0,+\infty)\times \mathbb{R}$ and satisfies
\begin{equation*}
\begin{cases}
\frac{\partial z_{1}(t,x)}{\partial t}\geq \left( \leq \right)
d_{1}\Delta
z_{1}(t,x)+r_{1}z_{1}(t,x)\left[ 1-z_{1}(t,x)+b_{1}z_{2}(t,x)\right] ,\\
\frac{%
\partial z_{2}(t,x)}{\partial t}\geq \left( \leq \right) d_{2}\Delta
z_{2}(t,x)+r_{2}z_{2}(t,x)\left[ 1-z_{2}(t,x)+b_{2}z_{1}(t,x)\right]
,\\
z_{1}(0,x)\geq \left( \leq \right) u_{1}(x),z_{2}(0,x)\geq \left(
\leq \right) u_{2}(x).
\end{cases}
\end{equation*}
Then $(z_1(t,x),z_2(t,x))\ge \left( \leq \right)
(u_1(t,x),u_2(t,x)).$
\end{lemma}

Now, we consider a nonautonomous equation as follows
\begin{equation}\label{2}
\begin{cases}
\frac{\partial u(t,x)}{\partial t}=d \Delta u(t,x)+f(t,x,u) ,\\
u(0,x)=u(x),
\end{cases}
\end{equation}
in which $f: \mathbb{R}\times \mathbb{R}\times \mathbb{R}^+\to
\mathbb{R}$ is assumed to be of class $C^{\delta/2, \delta}$ in
$(t,x),$ locally in $u$, for a given $\delta \in (0,1).$ Moreover,
$f$ is  also locally Lipschitz continuous in $u$ and of class $C^1$ in
$u\in [0, \beta]$ with $\beta>0$ uniformly with respect to $(t,x)\in
\mathbb{R}\times \mathbb{R},$ it is also supposed that $f(t,x,0)=0.$
Since \eqref{2} cannot generate a semiflow, the study of its asymptotic
spreading is very hard. To formulate its asymptotic spreading, we first present some important
definitions and results given by Berestycki et al. \cite[Section 1.5]{be}.

\begin{definition}\label{d2.1}{\rm
We say that \emph{complete spreading occurs} for a solution $u(t,x)$
of \eqref{2} if there is a function $t \to r(t) > 0 $ such that
$r(t)\to \infty $ as $t \to \infty$ and the family $(B_{r(t)})_{t\ge
0} $ is a family of propagation sets for $u$, that is
\[
\liminf_{t\to \infty}\left\{\inf_{x\in B_{r(t)}}u(t,x)\right\} >0,
\]
where $B_{r}=\{x\in \mathbb{R}: |x|<r\}.$}
\end{definition}

This definition, in fact, gives a description of the success of
spatial spreading/invasion, which is similar to the second item of Definition
\ref{d1.1}. Since there are only two directions in $\mathbb{R}$, we also show a
specific case of Berestycki et al. \cite[Definition 4]{be} as
follows.
\begin{definition}\label{d2.2}{\rm
We say that a family $(r(t))_{t\ge 0}$ of nonnegative real numbers
is \emph{a family of asymptotic spreading radii} for a solution
$u(t,x)$ of \eqref{2} if the family of segments $([-r(t),
r(t)])_{t\ge 0} $ is a family of propagation sets for $u(t,x)$, that
is
\[
\liminf_{t\to \infty}\left\{\inf_{s\in [0, r(t)]}u(t,\pm s)\right\}
>0.
\]
}
\end{definition}

\begin{definition}\label{d2.3}{\rm
We say that a family $(r(t))_{t\ge 0}$ is \emph{a family of
admissible radii} if $(r(t))_{t\ge 0}\in
C^{1+\delta/2}(\mathbb{R}^+,\mathbb{R}^+)$ and $\sup_{t\ge
0}|r'(t)|<  \infty$. }
\end{definition}
\begin{remark}\label{re}{\rm
If $\liminf_{t\to \infty}r(t)/t$ exists and is positive, then it is a lower bounds
of asymptotic  speed of spreading. Such a definition is still
useful because it can describe the phenomena of successful invasion, even if $\lim_{t\to
\infty}r(t)/t$ does not exist, see Berestycki et al. \cite[Section
1.8]{be} for the upper bounds of asymptotic  speed of spreading.}
\end{remark}

For $\phi \in C^{1,2}(\mathbb{R}\times \mathbb{R}),$ define
\[
L \phi=\frac{\partial \phi}{\partial t}- d\triangle \phi-
f'_u(t,x,0)\phi.
\]
Considering the \emph{generalized principal eigenvalue} problem
formulated by
\[
\lambda'_1=\inf\{\lambda \in \mathbb{R}, \exists \phi \in
C^{1,2}(\mathbb{R}\times \mathbb{R}) \bigcap
W^{1,\infty}(\mathbb{R}\times \mathbb{R}), \inf_{\mathbb{R}\times
\mathbb{R}}\phi >0, L \phi \le \lambda\phi\},
\]
then $\lambda'_1<0$ implies that the equilibrium $0$ is unstable and the following
conclusion holds.
\begin{lemma}[\cite{be}]\label{le2.4}
Assume that $\lambda'_1<0$ and there exists $r(t)$ of admissible
radii such that
\[
\liminf_{t\to \infty} u(t, \pm r(t)) >0.
\]
Then
\begin{equation}\label{st}
\liminf_{t\to \infty}\left\{\inf_{|x|\le r(t)}u(t,x)\right\}>0.
\end{equation}
\end{lemma}
\begin{lemma}[\cite{be}]\label{le2.5}
Let $u(t,x)$ be the solution of the Cauchy problem \eqref{2}
associated with an initial datum $u(x) >0.$ Assume that $\lambda'_1<0$ holds and there exists $r(t)$ of admissible
radii such that
\begin{equation}\label{9}
\liminf_{R\to \infty}\left\{\liminf_{t\to
\infty}\left\{\inf_{|x|<R}(4d f_u'(t, x\pm
r(t),0)-(r'(t))^2)\right\}\right\}>0.
\end{equation}
 Then \eqref{st}
holds for $u(t,x).$
\end{lemma}

\section{Main Results}
\noindent

In this section, we first prove the following result on asymptotic spreading.

\begin{theorem}\label{th1}
Let $(u_1(t,x),u_2(t,x))$ be defined by
\begin{equation}
\begin{cases}
\frac{\partial u_1(t,x)}{\partial t}=d_1\Delta
u_1(t,x)+r_1u_1(t,x)\left[
1-u_1(t,x)+b_1u_2(t,x)\right] , \\
\frac{\partial u_2(t,x)}{\partial t}=d_2\Delta
u_2(t,x)+r_2u_2(t,x)\left[ 1-u_2(t,x)+b_2u_1(t,x)\right],\\
u_1(0,x)=\phi_1(x), u_2(0,x)=\phi_2(x),
\end{cases}\label{3}
\end{equation}
in which $\phi_1(x)>0, \phi_2(x)>0$ are uniformly continuous and bounded for $x\in\mathbb{R}$. Suppose that $d_1r_1> d_2r_2$ holds. Then
\begin{equation}\label{4}
\lim_{t\to \infty}\inf_{|x|<ct}u_2(t,x)=\lim_{t\to
\infty}\sup_{|x|<ct}u_2(t,x)=k_2
\end{equation}
and
\begin{equation}\label{40}
\lim_{t\to \infty}\inf_{|x|<ct}u_1(t,x)=\lim_{t\to
\infty}\sup_{|x|<ct}u_1(t,x)=k_1
\end{equation}
for any $c<c^*=\min\{2 \sqrt{d_1r_1}, 2\sqrt{d_2r_2(1+b_2)}\}.$
\end{theorem}

For the main condition of the theorem, we give the following remark.
\begin{remark}{\rm
By \cite{aron1}, $d_1r_1 >d_2r_2$ implies that $u_1$ has stronger spreading ability than that of $u_2$ if the inter-specific actions disappear in \eqref{3} (namely, $b_1=b_2=0$ in \eqref{3}).}
\end{remark}

Before verifying Theorem \ref{th1}, we first prove several lemmas,
through which the conditions of Theorem \ref{th1} will be imposed.

\begin{lemma}\label{11}
For all $x\in \mathbb{R}, t>0,$ the Cauchy problem \eqref{3} admits
a unique solution $(u_1(t,x),u_2(t,x))$ such that
\[
(0,0)<(u_1(t,x),u_2(t,x))\le    (E_1, E_2),
\]
in which
\[
E_1=\max\left\{\sup_{x\in \mathbb{R} }\phi_1(x), k_1,
\frac{k_1}{k_2}\sup_{x\in \mathbb{R} }\phi_2(x)\right\},
E_2=\max\left\{\sup_{x\in \mathbb{R} }\phi_2(x), k_2,
\frac{k_2}{k_1}\sup_{x\in \mathbb{R} }\phi_1(x)\right\}.
\]
\end{lemma}
\begin{proof}
We prove the lemma by comparison principle. Clearly,
\[
r_1E_1(1-E_1+b_1E_2)\le 0, r_2E_2(1-E_2+b_2E_1)\le 0.
\]
Then $(E_1,E_2)$ is an upper solution  while $(0,0)$ is a lower
solution of \eqref{3} for $(t,x)\in (0,+\infty)\times \mathbb{R}$. Therefore, Lemma \ref{lecom} indicates that
\[
(0,0)\le(u_1(t,x),u_2(t,x))\le (E_1, E_2),\,\, (t,x)\in (0,+\infty)\times \mathbb{R}.
\]
The strict inequalities are evident by the following two facts:
\begin{description}
\item[(1)] The heat operator has property of infinite propagation
speed;
\item[(2)] $\phi_1(x), \phi_2(x)$ admit nonempty
supports.
\end{description}

The proof is complete.
\end{proof}

\begin{lemma}\label{12}
Define $c_1=2 \sqrt{d_1r_1}$. Suppose that $u_1(t,x)$ is defined by
\eqref{3}. Then
\begin{equation}\label{5}
\liminf_{t\to \infty}\inf_{|x|<ct}u_1(t,x)\ge 1
\end{equation}
for any $c<c_1.$
\end{lemma}
\begin{proof}
By Lemma \ref{11}, we see that
\[
\frac{\partial u_1(t,x)}{\partial t}\ge d_1\Delta
u_1(t,x)+r_1u_1(t,x)\left[ 1-u_1(t,x)\right],x\in \mathbb{R}, t>0.
\]
Then the result is evident by Lemmas \ref{l1.1} and \ref{com}. The
proof is complete.
\end{proof}

Let $\beta>0$ be a constant such that
\[
\beta u_1+r_1u_1[1-u_1+b_1u_2],\text{  }\beta
u_2+r_2u_2[1-u_2+b_2u_1]
\]
are monotone increasing if \[
(0,0)\le (u_1,u_2) \le    (E_1, E_2).
\]

For
$t\geq 0$, define $T(t)=(T_1(t),T_2(t))$ as follows
\[
\begin{cases}
T_1(t)u_1 (x)=\frac{e^{-\beta t}}{\sqrt{4\pi d_1t}}\int_{-\infty
}^\infty e^{-\frac{(x-y)^2}{4d_1t}}u_1 (y)dy, \\
T_2(t)u_2 (x)=\frac{e^{-\beta t}}{\sqrt{4\pi d_2t}}\int_{-\infty
}^\infty e^{-\frac{(x-y)^2}{4d_2t}}u_2 (y)dy.
\end{cases}
\]
Moreover, for $t\ge 0,s\ge 0,$ we still denote
\[
\begin{cases}
T_1(t)u_1 (s,x)=\frac{e^{-\beta t}}{\sqrt{4\pi d_1t}}\int_{-\infty
}^\infty e^{-\frac{(x-y)^2}{4d_1t}}u_1 (s,y)dy, \\
T_2(t)u_2 (s,x)=\frac{e^{-\beta t}}{\sqrt{4\pi d_2t}}\int_{-\infty
}^\infty e^{-\frac{(x-y)^2}{4d_2t}}u_2 (s,y)dy.
\end{cases}
\]

Let $X$ be defined as follows
\[
X=\{ u: u\text{ is a bounded and uniformly continuous function from
}\Bbb{R}\text{ to }\Bbb{R}^2\},
\]
which is a Banach space equipped with the supremum norm.
Then $T(t):X\to X$ is an analytic semigroup (see \cite{da}). Denote
\[
X^+=\{ u: u\in X, u\ge 0\}.
\]
Then $T(t):X^+\to X^+$ is a positive semigroup.
Using the standard theory of semigroup (see \cite{pazy}), we have the following
conclusion.
\begin{lemma}
The unique solution of \eqref{3} can also be formulated by
\begin{equation}\label{semi}
\begin{cases}
u_1(t,x)=T_1(t)\phi_1 (x)+ \int_0^t T_1(t-s)[F_1(u_1,u_2)](s,x)ds, \\
u_2(t,x)=T_2(t)\phi_2 (x)+ \int_0^t T_2(t-s)[F_2(u_1,u_2)](s,x)ds,
\end{cases}
\end{equation}
in which $F_1(u_1,u_2)=\beta u_1+r_1u_1[1-u_1+b_1u_2],\text{
}F_2(u_1,u_2)=\beta u_2+r_2u_2[1-u_2+b_2u_1].$
\end{lemma}

By above lemmas, we give the proof of Theorem \ref{th1} as
follows.

\begin{proof}
Let $\overline{c} <
c^*$ be fixed.  By \eqref{5}, we can choose $\epsilon >0$ satisfying the following facts.
\begin{description}
\item[(A)] There exists $T>0$ such that
\begin{equation}\label{7}
\inf_{4|x|<(\overline{c}+3c^*)t}u_1(t,x)> 1- \epsilon \text{ for all
}t>T.
\end{equation}
\item[(B)] $4\sqrt{d_2r_2(1+b_2(1-\epsilon))}>\overline{c}+c^* >2 \overline{c}$.
\end{description}

Define $r(t)=(\overline{c}+c^*)t/2.$ Then $r(t)\in
C^{\infty}(\mathbb{R}^+,\mathbb{R}^+ )$ such that Definition
\ref{d2.3} is true. Moreover, $\lim_{t\to \infty}r(t)= \infty $ also
implies that Definition \ref{d2.1} holds and a complete spreading of $u_1$
has been proved.

Denote
\begin{eqnarray}
\frac{\partial u_{2}(t,x)}{\partial t} &=&d_{2}\Delta
u_{2}(t,x)+r_{2}u_{2}(t,x)\left[ 1-u_{2}(t,x)+b_{2}u_{1}(t,x)\right]  \nonumber\\
&=&:d_{2}\Delta u_{2}(t,x)+\overline{f}(t,x,u_{2}),\label{*}
\end{eqnarray}
in which the definition of $\overline{f}$ is clear. To apply
Lemma \ref{le2.5}, we encounter some difficulties since
$\overline{f}(t,x,u_2)$ has no definition if $t<0.$ So we define $f$
such that
\[
f(t,x,u_2)=\begin{cases} \overline{f}(t,x,u_2), t>1,\\
g(t,x,u_2), t\in [0,1],\\
r_2u_2(1-u_2), t<0,
\end{cases}
\]
in which $g(t,x,u_2)=u_2g_1(t,x,u_2)$ with
\[
r_2(1-u_2)\le g_1(t,x,u_2)\le r_2 (1-u_2+b_2u_1), t\in [0,1], x\in\mathbb{R}
\]
such that $f$ satisfies the smooth condition of \eqref{2}. Since $[0,1]$
is a bounded interval, the existence of $f$ or $g$ is clear. Consider the
following initial value problem
\begin{equation*}
\begin{cases}
\frac{\partial z_{2}(t,x)}{\partial t} =d_{2}\Delta
z_{2}(t,x)+\overline{f}(t,x,z_{2}),\\
z_2(0,x)=z(x)>0,
\end{cases}
\end{equation*}
and
\begin{equation*}
\begin{cases}
\frac{\partial z(t,x)}{\partial t} =d_{2}\Delta
z(t,x)+f(t,x,z),\\
z(0,x)=z(x)>0.
\end{cases}
\end{equation*}
Then the comparison principle implies that \begin{equation}\label{z2} z_2(t,x)\ge z(t,x), \,\, (t,x)\in (0,\infty)\times \mathbb{R}.
\end{equation}

Thus, it suffices to study
\begin{eqnarray*}
\frac{\partial u_{2}(t,x)}{\partial t} =d_{2}\Delta u_{2}(t,x)+{f}(t,x,u_{2}).
\end{eqnarray*}
Evidently, for any $(t,x)\in\mathbb{R}\times \mathbb{R},$ we have
\begin{equation}\label{8}
r_2E_2\ge {f}_{u_2}'(t,x,0)\ge r_2.
\end{equation}

For the Fisher equation
\begin{eqnarray*}\label{13}
\frac{\partial u_{2}(t,x)}{\partial t} =d_{2}\Delta
u_{2}(t,x)+r_{2}u_{2}(t,x)\left[ 1-u_{2}(t,x)\right],
\end{eqnarray*}
we see that the corresponding  $\lambda_1'<0$ by Berestycki et al. \cite[Section 1.5]{be}. Then  \eqref{8} implies that the corresponding  $\lambda_1'$ of \eqref{*} is also negative.

Therefore, we just need to verify that \eqref{9} is true. For any
$R>0$, we see that
\begin{eqnarray*}
&&\liminf_{t\rightarrow \infty }\left\{ \inf_{|x|<R}\left[
4df_{u_{2}}^{\prime }(t,x\pm r(t),0)-(r^{\prime }(t))^{2}\right] \right\}  \\
&=&\liminf_{t\rightarrow \infty }\left\{ \inf_{|x|<R}\left[
4df_{u_{2}}^{\prime }(t,x\pm r(t),0)-\left( \frac{\overline{c}+c^{\ast }}{2}%
\right) ^{2}\right] \right\}  \\
&=&\liminf_{t\rightarrow \infty }\left\{ \inf_{|x|<R}\left[
4dr_{2}(1+b_{2}u_{1}(t,x\pm r(t)))\right] -\left(
\frac{\overline{c}+c^{\ast }}{2}\right) ^{2}\right\} .
\end{eqnarray*}%
By the item (B), it is clear that \eqref{7}  holds if $t>0$ is
large enough.
Therefore,
\begin{eqnarray*}
&&\liminf_{t\rightarrow \infty }\left\{ \inf_{|x|<R}\left[
4dr_{2}(1+b_{2}u_{1}(t,x\pm r(t)))\right] -\left(
\frac{\overline{c}+c^{\ast
}}{2}\right) ^{2}\right\}  \\
&\geq &4dr_{2}(1+b_{2}(1-\epsilon ))-\left( \frac{\overline{c}+c^{\ast }}{2}%
\right) ^{2} \\
&>&0
\end{eqnarray*}
since $R<-\frac{\overline{c}-c^{\ast }}{2}t \to \infty$ as  $t\to \infty.$

Note that $4dr_{2}(1+b_{2}(1-\epsilon
))-\frac{\overline{c}+c^{\ast }}{2}$ is independent of $R>0$, we
also obtain that
\[
\liminf_{R\rightarrow \infty }\left\{ \liminf_{t\rightarrow \infty
}\left\{
\inf_{|x|<R}\left[ 4df_{u_{2}}^{\prime }(t,x\pm r(t),0)-(r^{\prime }(t))^{2}%
\right] \right\} \right\} >0.
\]

By Lemmas \ref{le2.4}-\ref{le2.5} and \eqref{z2}, if
$c=\overline{c},$ then
\[
\liminf_{t\to \infty}\inf_{|x|<ct}u_2(t,x) >0.
\]
Due to the arbitrary of $\overline{c},$ what we have done implies that
\begin{equation*}
\liminf_{t\to \infty}\inf_{|x|<ct}u_2(t,x)>0,\,\, \liminf_{t\to \infty}\inf_{|x|<ct}u_1(t,x)>0
\end{equation*}
for any $c<c^*.$

It suffices to verify that \eqref{4} and \eqref{40} are also true for any fixed $\overline{c}< c^*$.
Let $2c_1=\overline{c}+ c^*, $ $c_1>c_2>\cdots > c_n >c_{n+1} >\cdots,$ $\lim_{n\to\infty}c_n=\overline{c},$ define positive constants
\begin{equation*}
\begin{cases}
\liminf_{t\to \infty}\inf_{|x|< {c_n}
t}u_1(t,x)=\underline{u}_1^n,\text{
}\liminf_{t\to \infty}\inf_{ |x|< {c_n} t}u_2(t,x)=\underline{u}_2^n,\\
\limsup_{t\to \infty}\sup_{|x|< {c_n}
t}u_1(t,x)=\overline{u}_1^n,\text{ }\limsup_{t\to \infty}\sup_{ |x|<
 {c_n} t}u_2(t,x)=\overline{u}_2^n.
\end{cases}
\end{equation*}
and
\begin{equation*}
\begin{cases}
\liminf_{t\to \infty}\inf_{|x|< \overline{c}
t}u_1(t,x)=\underline{u}_1,\text{
}\liminf_{t\to \infty}\inf_{ |x|< \overline{c} t}u_2(t,x)=\underline{u}_2,\\
\limsup_{t\to \infty}\sup_{|x|< \overline{c}
t}u_1(t,x)=\overline{u}_1,\text{ }\limsup_{t\to \infty}\sup_{ |x|<
\overline{c} t}u_2(t,x)=\overline{u}_2.
\end{cases}
\end{equation*}
Clearly, these positive constants are well defined and  satisfy
\begin{description}
\item[(L1)]$\underline{u}_1^n, \underline{u}_2^n$ are nondecreasing and $\underline{u}_1^n \le \underline{u}_1, \underline{u}_2^n \le \underline{u}_2$ for all $n>0$;
\item[(L2)] $\overline{u}_1^n, \overline{u}_2^n$ are nonincreasing and $\overline{u}_1^n \ge \overline{u}_1, \overline{u}_2^n \ge \overline{u}_2$ for all $n>0$;
\item[(L3)] $\lim_{n\to\infty}\underline{u}_i^n$ and $\lim_{n\to\infty}\overline{u}_i^n$ exist for $i=1,2;$
\item[(L4)] $\lim_{n\to\infty}\underline{u}_i^n\le \underline{u}_i \le \overline{u}_i \le \lim_{n\to\infty}\overline{u}_i^n,i=1,2.$
\end{description}

For each $n\ge 1,$ $t\to\infty  $ implies that  $ (c_{n+1}-c_n)t\to \infty$.
Using the positivity of the semigroup of $T(t)$ and the dominated convergence theorem for $t\to \infty$ in  \eqref{semi}, we see that
\[
\overline{u}_1^{n+1}\le \frac{\beta \overline{u}_1^{n}+r_1 \overline{u}_1^{n}[1-\overline{u}_1^{n}+b_1 \overline{u}_2^{n}]}{\beta}
\]
by the monotonicity of $F_1.$ Letting $n\to \infty,$ we further obtain that
\[
1-\lim_{n\to\infty}\overline{u}^n_{1}+b_{1}\lim_{n\to\infty}\overline{u}^n_{2} \geq 0.
\]
In a similar way, we have
\begin{eqnarray*}
1-\lim_{n\rightarrow \infty }\underline{u}_{1}^{n}+b_{1}\lim_{n\rightarrow
\infty }\underline{u}_{2}^{n} &\leq &0, \\
1-\lim_{n\rightarrow \infty }\underline{u}_{2}^{n}+b_{2}\lim_{n\rightarrow
\infty }\underline{u}_{1}^{n} &\leq &0, \\
1-\lim_{n\rightarrow \infty }\overline{u}_{2}^{n}+b_{2}\lim_{n\rightarrow
\infty }\overline{u}_{1}^{n} &\geq &0,
\end{eqnarray*}
and
\[
\lim_{n\to\infty}\underline{u}^n_{1}=\lim_{n\to\infty}\overline{u}^n_{1}=k_{1},\text{  }
\lim_{n\to\infty}\underline{u}^n_{2}=\lim_{n\to\infty}\overline{u}^n_{2}=k_{2}.
\]

From (L4), we obtain
\[
\underline{u}_{1}=\overline{u}_{1}=k_{1},\text{  }\underline{u}_{2}=\overline{u}%
_{2}=k_{2}.
\]
Since $\overline{c}$ is arbitrary, we complete the proof.
\end{proof}

We now present three remarks to further illustrate our conclusion.
\begin{remark}\label{r1}{\rm
If $d_1=d_2, r_1=r_2,$ then Lin et al. \cite{llm}
implies that \eqref{0} has a traveling wave solution connecting
$(0,0)$ with $(k_1,k_2)$ for any wave speed which is larger than
$2\sqrt{d_1r_1}=2\sqrt{d_2r_2}$. Therefore, if $0<\phi_1(x)<k_1, 0<\phi_2(x)<k_2$ admit compact supports, then the standard comparison
principle states that the asymptotic speeds of spreading of two invasion species are not
larger than $2\sqrt{d_1r_1}$ (see the subsequent Propositions
\ref{p1} and \ref{p2}). By Lemma \ref{12}, the asymptotic speeds of spreading of both invasion species
are  $2\sqrt{d_1r_1}=2\sqrt{d_2r_2}$.
}\end{remark}

\begin{remark}\label{r2}{\rm
If $d_1r_1>d_2r_2k_2=d_2r_2(1+b_2k_1)$ holds and
$0<\phi_1(x)<k_1, 0<\phi_2(x)<k_2$ admit compact supports,
 then
$0<u_i(t,x)\le k_i$ and
\[
\frac{\partial u_2(t,x)}{\partial t}\le d_2\Delta
u_2(t,x)+r_2u_2(t,x)\left[ k_2-u_2(t,x)\right].
\]
Namely, $u_2$ is a lower solution of the following  Cauchy problem
\begin{equation*}
\begin{cases}
\frac{\partial w_2(t,x)}{\partial t}= d_2\Delta
w_2(t,x)+r_2w_2(t,x)\left[ k_2-w_2(t,x)\right],\\
w_2(0,x)=\phi_2(x).
\end{cases}
\end{equation*}
Then the comparison principle (Lemma \ref{com}) indicates that
$u_2(t,x)\le w_2(t,x),$ and the upper bounds of asymptotic speed of spreading of
$u_2(t,x)$ is not
larger than $2\sqrt{d_2r_2k_2}$ by Lemma \ref{l1.1}. Recalling Lemma
\ref{12}, the lower bounds of asymptotic speed of spreading of $u_1(t,x)$ is
larger than $2 \sqrt{d_1r_1}$ such that two species have two
\emph{distinct} asymptotic speeds of spreading even if both of them are constants.}
\end{remark}

\begin{remark}\label{r3}{\rm
If $2 \sqrt{d_2r_2} < 2 \sqrt{d_1r_1}\le
2\sqrt{d_2r_2(1+b_2)}$ with $d_1=d_2$ and  $0<\phi_1(x)<k_1, 0<\phi_2(x)<k_2$ admit compact supports, then Lin et al. \cite{llm}
implies that the asymptotic speeds of spreading of both invasion species are less than $2
\sqrt{d_1r_1}$ (see Propositions \ref{p1} and \ref{p2}), and
Theorem \ref{th1} indicates that the asymptotic speeds of spreading of both species are   $2\sqrt{d_1r_1}$ such that the invasion of $u_2$ is fastened by $u_1$.
}
\end{remark}

Before ending this section, we also apply our main result to the study of traveling wave solutions of \eqref{0}.

\begin{proposition}\label{p1}
If \eqref{0} has a traveling wave solution
$(u_1(t,x),u_2(t,x))=(\psi_1(x+ct),\psi_2(x+ct))$ connecting $(0,0)$
with $(k_1,k_2).$ Then the asymptotic speeds of spreading  of $u_1(t,x),u_2(t,x)$
are not larger than $c$ if  $0<\phi_1(x)<k_1, 0<\phi_2(x)<k_2$ admit compact supports.
\end{proposition}
\begin{proof}
In the lemma, a traveling wave solution
$(u_1(t,x),u_2(t,x))=(\psi_1(x+ct),\psi_2(x+ct))$ connecting $(0,0)$
with $(k_1,k_2)$ is formulated by
\[
\lim_{s\to -\infty}(\psi_1(s),\psi_2(s))=(0,0), \lim_{s\to
\infty}(\psi_1(s),\psi_2(s))=(k_1,k_2),
\]
where $(\psi_1,\psi_2)$ is the wave profile and $c$ is the wave speed.

For any $\rho \in \mathbb{R},$ a traveling wave solution
$(\psi_1(x+ct+\rho),\psi_2(x+ct+\rho))$ is also an entire solution (defined for all $t\in\mathbb{R}$)
of the following Cauchy problem
\begin{equation*}
\begin{cases}
\frac{\partial u_1(t,x)}{\partial t}=d_1\Delta
u_1(t,x)+r_1u_1(t,x)\left[
1-u_1(t,x)+b_1u_2(t,x)\right] , \\
\frac{\partial u_2(t,x)}{\partial t}=d_2\Delta
u_2(t,x)+r_2u_2(t,x)\left[ 1-u_2(t,x)+b_2u_1(t,x)\right],\\
u_1(0,x)=\psi_1(x+\rho), u_2(0,x)=\psi_2(x+\rho).
\end{cases}
\end{equation*}
Letting $\rho $ large enough, then
\[
\psi_1(x+\rho)\ge \phi_1(x), \psi_2(x+\rho)\ge \phi_2(x)
\]
since  $0<\phi_1(x)<k_1, 0<\phi_2(x)<k_2$ have compact supports.
Now $(\psi_1(x+ct+\rho),\psi_2(x+ct+\rho))$ becomes an upper
solution of \eqref{3}. Then the comparison principle implies that
\[
\psi_1(x+ct+\rho)\ge u_1(t,x), \psi_2(x+ct+\rho)\ge u_2(t,x)
\]
and the result is clear.
\end{proof}

\begin{proposition}\label{p2}
Under the assumptions of Remark \ref{r1} or Remark \ref{r3}, \eqref{0} has a traveling wave solution
connecting $(0,0)$ with $(k_1,k_2)$ if $c> \max\{2\sqrt{d_1r_1}, 2
\sqrt{d_2r_2}\}$. Moreover, if $d_1\ge d_2, r_1\ge r_2,$ then \eqref{0} has a traveling wave solution
connecting $(0,0)$ with $(k_1,k_2)$ if $c> 2\sqrt{d_1r_1}$.
If $c<2\sqrt{d_1r_1},$ then \eqref{0} has not a traveling wave solution
connecting $(0,0)$ with $(k_1,k_2)$.
\end{proposition}
\begin{proof}
If $c> \max\{2\sqrt{d_1r_1}, 2 \sqrt{d_2r_2}\}$, then define
$0<\gamma _{i1}<\gamma _{i2}$ by
\begin{equation*}
d_i\gamma _{i1}^2-c\gamma _{i1}+r_i=d_i\gamma _{i2}^2-c\gamma _{i2}+r_i=0%
\text{ for }i=1,2.
\end{equation*}
By Lin et al. \cite[Theorem 5.11]{llm}, the lemma is true if
$(\gamma _{11},\gamma _{12})\bigcap (\gamma_{21},\gamma _{22})$ is
nonempty, which is evident if Remark \ref{r1} or Remark \ref{r3} holds or $d_1\ge d_2$ and $r_1\ge r_2$ are true.

The nonexistence of traveling wave solutions is clear by Lemmas \ref{l1.1} and \ref{12}, and we omit the proof here. The proof is
complete.
\end{proof}

\section{Discussion}
\noindent

In ecological systems, the cooperatitive/symbiotic/mutualistic
communities are very universal.  For example, the role that insects, in
particular bees, have in the fecundation of flowers, see Boucher
\cite{bou}. Furthermore, Malchow et al. \cite[Section 4.3.2]{mal}
also introduced many examples. In population dynamics, the behavior
of many cooperative kinetic systems is very simple: If a
cooperative system admits only one positive equilibrium, then the
equilibrium is asymptotic stable and the others are unstable.
These mathematical results are very easy and can be found in many
textbooks, we also refer to Malchow et al. \cite{mal}, Murray
\cite{murray}. In particular, if $b_1b_2<1$ in \eqref{0}, then $(k_1,k_2)$ is
stable and the phase plane of the corresponding kinetic system is
very clear, see Murray \cite[pp. 101]{murray}. Biologically,
$(k_1,k_2)>(1,1)$ implies that each species has increased its steady
state population from its maximum value in isolation \cite{murray}, which is achieved by inter-specific cooperation.

However, when the spatial-temporal structure is involved in
cooperative systems, e.g., the spatial dispersal of plant and seeds
(see Murray \cite[Section 3.6]{murray}), its dynamical properties
may be very complex since the process often involves the far-from-equilibrium dynamics.
By Liang and Zhao \cite{liang}, Lui \cite{lui1,lui2},
if an irreducible cooperative system  admits two steady states, it is very likely that
different unknown functions have the same asymptotic speed of spreading.
However, Remarks \ref{r1}-\ref{r3} show that the complex propagation modes of evolutionary systems with
multi equilibria since it is necessary to
formulate the asymptotic spreading of different unknown functions by
different indices. Note that the number of steady states is
determined by the nonlinearities,  this certainly indicates the
complex arising from the nonlinearities.

We now consider the linear determinacy problem. Because we consider
the spatial invasion of two species, then one interesting
equilibrium is $(0,0)$ that is invadable. If the asymptotic speeds of spreading are linearly
determinate, then the asymptotic speeds of spreading will be fully  determined by
$d_1,r_1,d_2,r_2,$ which is impossible by Remarks \ref{r1}-\ref{r3}. Therefore, our
results show the effect of inter-specific cooperation from the
following two factors: (1) asymptotic speed of spreading or its lower
bounds of $u_2$ since $c^*> 2\sqrt{d_2r_2}$ in Theorem \ref{th1}; (2)
eventual population densities on the coexistence domain because of
$(k_1,k_2)> (1,1)$.

In this paper, utilizing the theory established by Berestycki
\cite{be}, we obtain some estimates of the asymptotic speeds of spreading, which
partly shows the role of nonlinearity. Unfortunately, only the lower
bounds and upper bounds of asymptotic speeds of spreading are obtained, precise results need
further investigation.

\section*{Acknowledgments}
\noindent

I would like to express my gratitude to the anonymous referee for a careful reading and helpful suggestions which led to an improvement of my original manuscript. This work was supported  by NSF of China (11101094), NSF of Gansu
Province of China (096RJZA051) and Fundamental Research Funds for the Central Universities (lzujbky-2010-67).


\begin{thebibliography}{99}
\bibitem{aron1} D.G. Aronson, H.F. Weinberger, Nonlinear diffusion in
population genetics, combustion, and nerve pulse propagation, In:
\emph{Partial Differential Equations and Related Topics} (Ed. by
J.A. Goldstein), Lecture Notes in Mathematics, Vol. 446, pp. 5-49,
Springer, Berlin, 1975.

\bibitem{be} H. Berestycki, F. Hamel, G. Nadin, Asymptotic spreading in heterogeneous diffusive
excitable media, \emph{J. Funct. Anal.,} \textbf{255} (2008),
2146-2189

\bibitem{be2} H. Berestycki, F. Hamel, N. Nadirashvili,
The speed of propagation for KPP type problems. II. General domains,
\emph{J. Amer. Math. Soc.,} \textbf{23} (2010),  1-34.

\bibitem{bo} F. van den Bosch, J.A.J. Metz, O. Diekmann,  The velocity of
spatial population expansion, \emph{J. Math. Biol.,} \textbf{28}
(1990), 529-565.

\bibitem{bou} D.H. Boucher, \emph{The Biology of Mutualism: Ecology and
Evolution}, Croom Helm, London, 1985.

\bibitem{da} D. Daners, P.K. Medina, \emph{Abstract Evolution Equations, Periodic Problems and Applications,}
Longman Scientific \& Technical, Harlow,
1992.

\bibitem{diekmann} O. Diekmann, Thresholds and traveling waves for the
geographical spread of infection, \textit{J. Math. Biol.,}
\textbf{6} (1978), 109-130.

\bibitem{diek} O. Diekmann, Run for your life. A note on the asymptotic speed of
propagation of an epidemic, \emph{J. Differential Equations,}
\textbf{33} (1979), 58-73.

\bibitem{hr} K.P. Hadeler, F. Rothe,  Travelling fronts in nonlinear diffusion equations,
\emph{J. Math. Biol.,} \textbf{2} (1975), 251-263.

\bibitem{hsuzhao} S.B. Hsu, X.Q. Zhao, Spreading speeds and traveling
waves for nonmonotone integrodifference equations, \emph{SIAM J.
Math. Anal.,} \textbf{40} (2008), 776-789.

\bibitem{liwe} B. Li, H.F. Weinberger, M.A. Lewis, Spreading speeds as
slowest wave speeds for cooperative systems, \emph{Math. Biosci.,}
\textbf{196} (2005), 82-98.

\bibitem{liang} X. Liang, X.Q. Zhao,
Asymptotic speeds of spread and traveling waves for monotone
semiflows with applications, \emph{Comm. Pure Appl. Math.,}
\textbf{60} (2007), 1-40.

\bibitem{lin} G. Lin, Spreading speeds of a Lotka-Volterra predator-prey system: The role
of the predator, \emph{Nonlinear Analysis TMA}, \textbf{74} (2011), 2448-2461.

\bibitem{llm} G. Lin, W.T.  Li, M.  Ma, Travelling wave solutions in
delayed reaction diffusion systems with  applications to
multi-species models, \emph{Discrete Contin. Dyn. Syst. Ser. B,}
\textbf{19} (2010), 393-414.

\bibitem{llr1} G. Lin, W.T. Li, S. Ruan, Spreading speeds and traveling waves in competitive
recursion systems, \emph{J. Math. Biol.,} \textbf{62} (2011), 165-201.

\bibitem{lui1} R. Lui,  Biological growth and spread modeled
by systems of recursions. I. Mathematical theory,  \emph{Math.
Biosci.,} \textbf{93} (1989), 269-295.

\bibitem{lui2} R. Lui,  Biological growth and spread
modeled by systems of recursions. II. Biological theory, \emph{Math.
Biosci.,} \textbf{107} (1991), 255-287.

\bibitem{mal} H. Malchow, S.V. Petrovskii, E. Venturino,
\emph{Spatiaotemporal Patterns in Ecology and Epidemiology: Theory,
Models and Simulation}, Chapman \& Hall/CRC, Boca Raton, 2008.

\bibitem{moli} D. Mollison,  Dependence of epidemic and population
velocities on basic parameters,  \emph{Math. Biosci.,} \textbf{107}
(1991), 255-287.

\bibitem{murray}  J.D. Murray, \emph{Mathematical Biology}
(The Third Edition), Springer-Verlag, New York, 2002.

\bibitem{pao} C.V. Pao,  \emph{Nonlinear Parabolic and Elliptic Equations},
Plenum, New York, 1992.

\bibitem{pazy}  A. Pazy,  \emph{Semigroups of Linear Operators and
Applications to Partial Differential Equations,} Springer-Verlag,
New York, 1983.

\bibitem{shi} N. Shigesada, K.  Kawasaki, \emph{Biological Invasions: Theory and Practice,} Oxford University Press, Oxford, 1997.

\bibitem{smoller} J. Smoller, \emph{Shock Waves and Reaction-Diffusion
Equations}, Springer-Verlag, New York, 1994.

\bibitem{th1} H.R. Thieme,
Asymptotic estimates of the solutions of nonlinear integral
equations and asymptotic speeds for the spread of populations,
\emph{J. Reine Angew. Math.,} \textbf{306} (1979), 94-121.

\bibitem{thieme} H.R. Thieme, X.Q. Zhao,
Asymptotic speeds of spread and traveling waves for integral
equations and delayed reaction diffusion models, \emph{J.
Differential Equations,} \textbf{195} (2003), 430-470.

\bibitem{weinberger} H.F. Weinberger, M.A. Lewis, B. Li,
Analysis of linear determinacy for spread in cooperative models,
\emph{J. Math. Biol.,}  \textbf{45} (2002), 183-218.

\bibitem{wll} H.F. Weinberger, M.A. Lewis, B. Li,
Anomalous spreading speeds of cooperative recursion systems,
\emph{J. Math. Biol.,}  \textbf{55} (2007), 207-222.

\bibitem{yeli} Q. Ye, Z. Li, \emph{Introduction to Reaction Diffusion
Equations}, Science Press, Beijing, 1990.
\end{thebibliography}
\end{document}